\documentclass{amsart} 
\usepackage[latin1]{inputenc}
\usepackage{graphicx}
\usepackage[english]{babel}
\usepackage{amssymb,enumerate,amsmath}
\usepackage[latin1]{inputenc}
\usepackage[T1]{fontenc}
\usepackage{ae,aecompl}
\usepackage{tikz}
\usepackage{array}
\usepackage{mathrsfs}
\usepackage{color}

\usepackage[pdftex,bookmarks,colorlinks, citecolor=blue]{hyperref} 
\usepackage[numbers]{natbib}
\newcommand{\Gr}{Gr\"obner}
\newcommand{\Pl}{Pl\"ucker}
\newcommand{\rk}{\mathnormal{rk}\,}
\newcommand{\sat}{{\textnormal{sat}}}
\newcommand{\reg}{\textnormal{reg}}
\newcommand{\PGL}{\textsc{pgl}}

\newcommand{\HilbFunctor}[2]{\underline{\mathbf{Hilb}}^{#2}_{#1}}
\newcommand{\HilbFunctorR}[3]{\underline{\mathbf{Hilb}}^{#2,[#3]}_{#1}}
\newcommand{\Kalg}{\underline{k\textnormal{-Alg}}}
\newcommand{\Sets}{\underline{\textnormal{Set}}}
\newcommand{\schm}{\underline{\textnormal{Sch}/k}}
\newcommand{\HilbScheme}[2]{\mathbf{Hilb}^{#2}_{#1}}
\newcommand{\HilbSchemeR}[3]{\mathbf{Hilb}^{#2,[#3]}_{#1}}
\newcommand{\PP}{\mathbb{P}}

\newcommand{\supp}{\mathrm{Supp}\,}

\newcommand{\Ht}{\mathrm{Ht}\,} 
\newcommand{\cN}{\mathcal N} 
\newcommand{\Proj}{\textnormal{Proj}\,}
\newcommand{\Spec}{\textnormal{Spec}\,}
\newcommand{\NN}{\mathbb{N}}
\newcommand{\QQ}{\mathbb{Q}}
\newcommand{\GrassFunctor}[2]{\underline{\mathbf{Gr}}_{#1}^{#2}}
\newcommand{\GrassScheme}[2]{\mathbf{Gr}_{#1}^{#2}}

\renewcommand{\a}{\mathrm{a}}
\renewcommand{\b}{\mathrm{b}}
\newcommand{\RevLex}{\mathtt{DegRevLex}}
\newcommand{\MFFunctor}[1]{\underline{\mathbf{Mf}}_{#1}}

 \numberwithin{equation}{section}

\newtheorem{lemma}{Lemma}[section] 
\newtheorem{theorem}[lemma]{Theorem} 
\newtheorem{corollary}[lemma]{Corollary} 
\newtheorem{prop}[lemma]{Proposition} 

\newenvironment{proof1}{\noindent{\it Proof of Theorem \ref{principale1intro}}.}{\qed}
\newenvironment{proof2}{\noindent{\it Proof of Theorem \ref{i1}}.}{\qed}

\theoremstyle{definition}
\newtheorem{notation}[lemma]{Notation} 
\newtheorem{remark}[lemma]{Remark} 
\newtheorem{definition}[lemma]{Definition} 
\newtheorem{example}[lemma]{Example}

\DeclareMathAlphabet{\mathpzc}{OT1}{pzc}{m}{it}

\begin{document}

\title[Points of the Hilbert scheme with bounded regularity]{The locus of points of the Hilbert scheme with bounded regularity}

\author[E.~Ballico]{Edoardo Ballico}
\address{Dipartimento di Matematica dell'Universit\`{a} di Trento\\ 
         38123 Povo (TN), Italy}
\email{\href{mailto:ballico@science.unitn.it}{ballico@science.unitn.it}}

\author[C.~Bertone]{Cristina Bertone}
\address{Dipartimento di Matematica dell'Universit\`{a} di Torino\\ 
         Via Carlo Alberto 10, 
         10123 Torino, Italy}
\email{\href{mailto:cristina.bertone@unito.it}{cristina.bertone@unito.it}}

\author[M.~Roggero]{Margherita Roggero}
\address{Dipartimento di Matematica dell'Universit\`{a} di Torino\\ 
         Via Carlo Alberto 10, 
         10123 Torino, Italy}
\email{\href{mailto:margherita.roggero@unito.it}{margherita.roggero@unito.it}}

\thanks{The third author was supported by the framework of PRIN 2010-11 \emph{Geometria delle variet\`a algebriche}, cofinanced by MIUR}



\begin{abstract} 
In this paper we consider the Hilbert scheme $\HilbScheme{p(t)}{n}$ parame\-te\-rizing subschemes of $\PP^n$ with Hilbert polynomial $p(t)$, and we 
investigate its
locus containing points corresponding to schemes with regularity lower than or equal to a fixed integer $r'$. This locus is an open 
subscheme of $\HilbScheme{p(t)}{n}$ 
and, for every $s\geq r'$, we describe it as a locally closed  subscheme of the Grasmannian $\GrassScheme{p(s)}{N(s)}$ given by a set of equations of degree $\leq \deg(p(t))+2$ and 
linear inequalities in the coordinates of the  Pl\"ucker embedding.
\end{abstract}


\maketitle

\section{Introduction}

One of the most interesting and investigated projective schemes in Algebraic Geometry is the \emph{Hilbert scheme}.
We consider a projective space $\PP^n$
over  a field $k$ of characteristic 0, {we fix an \emph{admissible} Hilbert polynomial $p(t)$, 
 i.e.  the Hilbert
polynomial of some subscheme of $\PP^n$}; the Hilbert scheme $\HilbScheme{p(t)}{n}$ parameterizes all the subschemes $X\subseteq \PP^n$
having $p(t)$ as Hilbert polynomial
of its coordinate ring. 

Even though the Hilbert scheme was formally introduced by \citet{Gro} in the 60's, it is not completely  understood yet.

The set of all schemes $X\in \HilbScheme{p(t)}{n}$ has bounded regularity  (in the sense of
 Castelnuovo-Mumford regularity \citep[Definition at page 99]{m}).  Mumford   proved     it to give another proof
of the representability of the Hilbert scheme \citep[Theorem at page 101]{m}. Then Gotzmann found the optimal upper bound
for the regularity (\citet[78, Satz (2.9)]{gotz}; see also \citet[Lemma
 C.23]{ik}). By Gotzmann's Regularity theorem  there
 exists a number $r$ only depending on $p(t)$ and easily obtained from the coefficients of $p(t)$, called \emph{Gotzmann number}, for
 which the ideal sheaf $\mathcal I_X$ of
 each scheme $X$ in $\HilbScheme{p(t)}{n}$  is $r$-regular. We refer to  \citet{ik} or \citet[Theorem 3.1]{gotz} for the explicit value of 
 the Gotzmann number $r$ for any Hilbert polynomial $p(t)$.
 
Let  $N(t)$ be the dimension of the $k$-vector space $H^0\mathcal O_{\PP^n} (t) $.  If we choose $t=r$, 
 the canonical map  $H^0\mathcal O_{\PP^n} (r) \rightarrow H^0\mathcal O_X (r)$  is surjective and the dimension of 
 $ H^0\mathcal O_X (r)$ as a $k$-vector space is $p(r)$. This is the starting point of the  classical construction that allows to embed the Hilbert scheme 
 as a closed subscheme of the Grassmannian   $\GrassScheme{p(r)}{N(r)}$.

In many cases the regularity  
of a schemes $X\in \HilbScheme{p(t)}{n}$ we are interested in is  far lower than the Gotzmann number $r$.

\begin{example}\label{esintro}
In each one of the following cases,  we get a bound $r'$ on the regularity of the subscheme $X$ just applying Castelnuovo-Mumford's lemma 
and  the references
quoted in each item:
\begin{itemize}
\item[(i)] Consider $p(t) =c$. The Gotzmann number $r$ of $p(t)$ is $c$. Let $X\subset \mathbb {P}^n$ be $c$ general points. 
If $c\gg n$ then we may take
$r' \sim (n!\cdot c)^{1/n}$ \citep{hirscho}.
\item[(ii)] Consider $n=g-1\ge 2$ and $p(t) = (2g-2)t+1-g = 2nt-n$. Let $X\subset \mathbb {P}^n$ be a smooth canonically embedded
curve of genus $n+1$. 
We may take $r'=3$ \citep[Corollary 9.4]{e}.
\item[(iii)] Fix an integer $g$ such that $0\le g \le n-1$ and consider a linearly normal smooth curve $X\subset \mathbb {P}^n$ 
of degree $n+g$; hence 
$p(t) = (n+g)t+1-g$. $X$ is arithmetically Cohen-Macaulay and we may take $r'=3$ (\citet[Theorem 8.1 and Corollary 8.2]{e}).
\item[(iv)] Let $X\subset \mathbb {P}^n$ be an integral and non-degenerate curve of degree $c$. Hence $p(t)=ct+1-g$, where $g:= p_a(X)$. 
We may take 
$r' = c+2-n$ (\citet{glp}, \citet[Theorem 5.1]{e})
 \end{itemize}

For many other higher dimensional examples, see several papers on the Eisenbud-Goto conjecture (see \citet{e}, 
Conjecture 5.2 and the references therein).
\end{example}

By the semicontinuity theorem,  the locus of points of the Hilbert scheme having regularity upper bounded by a fixed integer $r'$ is an open subscheme of 
$\HilbScheme{p(t)}{n}$, hence a locally closed subscheme of $\GrassScheme{p(r)}{N(r)}$. Moreover,  this subscheme is  connected, just like the whole Hilbert scheme
(see \citet{h2}, \citet[Section 5.2]{mall}).

 This locus of $\HilbScheme{p(t)}{n}$  is the main object of study of the present paper.
 To this aim, we define the functor $\HilbFunctorR{p(t)}{n}{r'}$ which associates to each    scheme $Z$ over $k$,  the set of schemes $X$ in $\PP^n\times_k Z$ 
 such that the projection over $Z$ is flat and whose fibers have  Hilbert polynomial $p(t)$ and 
regularity $\leq r'$ (see Definition \ref{def:hilbsubfunctorRing}). For $r'=r$, the functor $\HilbFunctorR{p(t)}{n}{r}$ coincides with the classical  Hilbert 
functor $\HilbFunctor{p(t)}{n}$.

We embed $\HilbFunctorR{p(t)}{n}{r'}$ in a suitable Grassmann functor  $\GrassFunctor{p(s)}{N(s)}$, for $s\geq r'$, similarly to the  embedding of the Hilbert functor  in the Grassmann functor $\GrassFunctor{p(r)}{N(r)}$ \citep{Gro,ik,HaimSturm}. 
We prove that $\HilbFunctorR{p(t)}{n}{r'}$ is representable, using suitable open covers of these functors, obtained exploiting the properties of Borel fixed ideals 
and the action of the linear group $\PGL(n+1)$. In this setting, a special role is played by  a closed subset of the Grassmannian $\GrassScheme{p(s)}{N(s)}$ that 
we denote by ${L}^{[r',s]}_{p(t)}$,  which is the intersection
of $\GrassScheme{p(s)}{N(s)}$ and a linear space under the \Pl\ embedding  (see Notation \ref{not}).

As pointed out by the referee,  our arguments allow to give another proof of the existence of the Hilbert scheme over a field with characteristic zero, not relying 
on the flattening stratification. This is a very nice remark (due to the referee). The crucial technical tools  are marked bases and schemes \citep{CR,BCLR,BLR,LRFunt}.
The first main result of the paper is the following.
\begin{theorem} \label{principale1intro}The Hilbert functor with bounded regularity  $\HilbFunctorR{p(t)}{n}{r'}$ is representable. For every  $s\geq r'$, 
the representing scheme $\HilbSchemeR{p(t)}{n}{r'}$ can be embedded in  $\GrassScheme{p(s)}{N(s)}\setminus {L}^{[r',s]}_{p(t)}$
 as a  closed subscheme.
Moreover, 
\begin{enumerate}[(i)]
\item \label{principale_i}for $s=r'=r$, $\HilbFunctor{p(t)}{n}=\HilbFunctorR{p(t)}{n}{r}$ is representable and $\HilbScheme{p(t)}{n}$ 
can be embedded as closed subscheme of  $\GrassScheme{p(r)}{N(r)}$;
\item for $r'<r$,  $\HilbSchemeR{p(t)}{n}{r'}$ is the open subscheme  $\HilbScheme{p(t)}{n}\cap (\GrassScheme{p(r)}{N(r)}\setminus {L}^{[r',r]}_{p(t)})$.
\end{enumerate}
\end{theorem}

By Gotzmann's Persistence Theorem (\citet{gotz}), explicit equations for the scheme structure of $\HilbScheme{p(t)}{n}$ in  $\GrassScheme{p(r)}{N(r)}$ 
can be found computationally, 
imposing  vanishing conditions on some minors of a particular matrix.  By the \Pl \ embedding, $\HilbScheme{p(t)}{n}$ becomes 
a subscheme of $\PP^{E(r)-1}$, for $E(t):=\binom{N(t)}{p(t)}$.   The number $E(r)$ is generally huge, and  even larger is the 
number of minors to consider. Nonetheless,  many authors dealt with 
the challenge of finding 
 sets of
defining equations 
for $\HilbScheme{p(t)}{n}$ in the Grassmannian.
For instance \citet{ik},  \citet{B} in his Ph.D. Thesis, \citet{HaimSturm} and more recently \citet{ABM}  and \citet{BLMR}  
find  sets  of equations of different degrees. 
The least degree bound among these is that  in \citet{BLMR}, where the authors find a set of  equations
of degree $\le d+2$, with 
$d:=\deg (p(t))$.

\begin{example}
 If $d=0$, i.e. $p(t) = c$, with $c$ a positive integer, then $r=c$ and in this case we get $E=\binom{N(c)}{c}$.

 For $d>0$, $r$ is a more complicated function depending on the coefficients of $p(t)$. For instance, if $p(t) = at+b$, the admissibility 
 of $p(t)$ means $a>0$ 
 and $b \ge -a(a-3)/2$; in this case $r= a(a-1)/2+b$ (\citet[Example 1.2]{CLMR}). Hence, $E=\binom{N(r)}{p(r)}$ can be a very large number
 even for quite small
 values of $a$ and $b$.    In this case $d=1$, hence by \citet{BLMR}, $\HilbScheme{p(t)}{n}$  is a subscheme of the Grassmannian 
 $\GrassScheme{p(r)}{N(r)}$
 defined by a set of equations of degree $\leq 3$. 
\end{example}

As second achievement of this paper, we exhibit a set of defining equations for $\HilbSchemeR{p(t)}{n}{r'}$  in $\GrassScheme{p(s)}{N(s)}\setminus L^{[r',s]}_{p(t)}$. 

\begin{theorem}\label{i1}
The Hilbert scheme with bounded regularity $ \HilbSchemeR{p(t)}{n}{r'} $ can be embedded as a closed subscheme 
of  $\GrassScheme{p(s)}{N(s)}\setminus L^{[r',s]}_{p(t)}$
defined by polynomial equations  of degree $\leq \deg(p(t))+2$ in the coordinates of the  Pl\"ucker embedding 
of $\GrassScheme{p(s)}{N(s)}$.
\end{theorem}

We obtain this exploiting again the properties of marked bases and schemes, the action of $\PGL(n+1)$ and also 
the theory of extensors of exterior algebras over a ring, developed in \citep{BLMR}.

In the case $s=r$, 
Theorem \ref{i1} gives a nice addition to \citet{BLMR}. Furthermore, if $s < r$, $E(s)\ll E(r)$  and in particular
for $s=r'<r$, our equations  involve a much smaller number of variables than is needed for the whole Hilbert scheme.


\section{Hilbert functors with bounded regularity}\label{sec:generalitiesHilb}

 In the following we  fix a field $k$. All rings will be  Noetherian $k$-algebras 
and all the schemes will be locally  Noetherian schemes over $k$. 
Moreover, we fix an integer $n$ and an admissible Hilbert polynomial $p(t)$  for projective subschemes in the $n$-dimensional projective space $\PP^n$
over the field $k$,   and let $N(t):=\binom{n+t}{n}$, $q(t):=N(t)-p(t)$ and   $r$ be the Gotzmann number of the polynomial $p(t)$.

We will denote  by $P$ the polynomial ring $ k[x_0,\ldots,x_n]$, so that $\PP^n =\Proj P$. For any  
$k$-algebra $A$,  $P \otimes_k A$ is the polynomial ring $A[x_0,\ldots,x_n] $ and  $\PP^n_A$ is the projective space 
$\Proj P \times_k\Spec A$. For any Noetherian $k$-scheme $S$ set $\PP ^n_S:= \Proj P\times _k S$.
If  $A=K$  is a field,  a closed subscheme    $X\subset \PP^n_K$    is said to be $t$-regular if its ideal sheaf $\mathcal {I}_X$ is $t$-regular, i.e 
$H^{i}(\mathcal {I}_X(t-i))=0$ for all $i>0$. 
Castelnuovo-Mumford's lemma gives that $t$-regularity implies $t'$-regularity for all $t'>t$, i.e. $H^{i}(\mathcal {I}_X(t-j))=0$ for
all $i>0$ and all $j\le i$ (\citet[p. 504, 516]{e2}).

Fix a Noetherian $k$-scheme $S$. A subscheme $X \subset \PP^n _S$ flat over $S$ is $t$-regular if its ideal sheaf 
$\mathcal{I}_X$ is $t$-regular, i.e. if  for each field $K$ and each map $\Spec K \to S$ we have $H^{i}(\mathcal{I}_{X_{K}} (t-i)) = 0,\ \forall\ i > 0$. 
Since $t$-regularity
implies $(t+1)$-regularity,  we may apply the base change theorem for cohomology (\citet[III.12.11]{AG}, or \citet[part (3) of Theorem 5.10]{nit}) 
and see that if $X$ is $t$-regular, then
$H^{i}(X,\mathcal {I}_X(t-j)) =0$ for all $i >0$ and all $j\le i$. Since we defined $t$-regularity using fields, 
it is obvious that if $X\subset \PP^n_S$ is $t$-regular, then
for each pull-back map $f: S' \to S$, the scheme $f^\ast (X)$ is $t$-regular.

The semicontinuity theorem for cohomology
gives that the set of all $s\in S$ such that $X_s$ is $t$-regular is open in $S$ (\citet[part (1) of  Theorem 5.10 ]{nit}, or \citet[III.12.8]{AG}). 
 The lowest integer $x$ such that  for each $s\in S$ the scheme $X_s$ is $x$-regular is called the \emph{Castelnuovo-Mumford regularity} of $X$.

In the following, $\HilbFunctor{p(t)}{n}$ denotes the Hilbert functor 
\begin{equation*}
\HilbFunctor{p(t)}{n}: \schm^{\circ}\rightarrow \Sets.
\end{equation*}
It associates to any object $Z$ of the category of schemes over $k$ the set
\[
 \HilbFunctor{p(t)}{n}(Z) = \left\{X \subset \PP^n \times_k Z\ \left\vert \ \begin{array}{l}
 \ X \rightarrow Z \text{ flat and whose fibers } \\ \text{ have  Hilbert
 polynomial } p(t)  \end{array}\right.\right\}
\]
and to any morphism of schemes $f:Z \rightarrow Z'$ the map 
\[
\begin{split}
\HilbFunctor{p(t)}{n}(f):&\ \HilbFunctor{p(t)}{n}(Z') \rightarrow \HilbFunctor{p(t)}{n}(Z)\\
&\parbox{2.18cm}{\centering $X'$} \mapsto\ X' \times_{Z'} Z
\end{split}
\]

\begin{remark}
The Hilbert  functor is representable, i.e. it is isomorphic to the functor 
of points of a scheme over 
$k$: the scheme representing $\HilbFunctor{p(t)}{n}$ is called Hilbert scheme
and denoted  by $\HilbScheme{p(t)}{n}$. The existence of this scheme was 
first proved by \citet{Gro}, using flattening stratifications. As suggested by 
the referee, in the present paper we do not use the representability of the Hilbert
scheme as a tool, but we give a new independent and self-contained proof that uses 
the properties of marked bases (Section \ref{sec:generalSetting}) instead of  flattening 
stratification or Fitting ideals \citep{HaimSturm}.
\end{remark}

This paper mainly concerns the following functor.
\begin{definition}\label{def:hilbsubfunctorRing}
For any fixed integer $r'$,  we denote by $\HilbFunctorR{p(t)}{n}{r'}$ the functor
\begin{equation*}
\HilbFunctorR{p(t)}{n}{r'}: \schm^{\circ}\rightarrow \Sets.
\end{equation*}
It associates to any object $Z$ of the category of schemes over $k$ the set
\begin{equation*}
 \HilbFunctorR{p(t)}{n}{r'}(Z) = \left\{X \subset \PP^n \times_k Z\ \left\vert \ \begin{array}{l}
 \ X \rightarrow Z \text{ flat and whose fibers } \\ \text{ have  Hilbert
 polynomial } p(t)  \\
\text{ and regularity}\leq r'\end{array}\right.\right\}.
\end{equation*}
and to any morphism of schemes $f:Z \rightarrow Z'$ the map 
\[
\begin{split}
\HilbFunctorR{p(t)}{n}{r'}(f):&\ \HilbFunctorR{p(t)}{n}{r'}(Z') \rightarrow \HilbFunctorR{p(t)}{n}{r'}(Z)\\
&\parbox{2.18cm}{\centering $X'$} \mapsto\ X' \times_{Z'} Z
\end{split}
\]
We will say that  $\HilbFunctorR{p(t)}{n}{r'}$  is the   \emph{Hilbert functor with bounded regularity $r'$ and Hilbert polynomial $p(t)$ of $\PP^n$}.
\end{definition}

\begin{remark}\label{zariski}
 By Gotzmann's Regularity theorem,  the Gotzmann  number $r$
 only depends on $p(t)$ and the ideal sheaf $\mathcal I_X$ of
 each scheme $X$ in $\HilbFunctor{p(t)}{n}(Z)$  is $r$-regular, for every $k$-scheme $Z$. 
Hence, by Definition \ref{def:hilbsubfunctorRing}, $\HilbFunctor{p(t)}{n} =\HilbFunctorR{p(t)}{n}{r} $. 

Furthermore, by the semicontinuity theorem for regularity, for every $r''<r'$,
$\HilbFunctorR{p(t)}{n}{r''}$ can be considered as an open subfunctor of $\HilbFunctorR{p(t)}{n}{r'}$. 
Obviously, $\HilbFunctor{p(t)}{n}$ is a Zariski sheaf (it is even a sheaf for finer topologies \citep[Section 5.1.3]{nit}). 

Summing up these two facts, $\HilbFunctorR{p(t)}{n}{r'}$ is a Zariski sheaf for every $r'$, and this allows us to apply \citep[Lemma E.11]{sernesi} 
and consider the Hilbert functor with bounded regularity defined between the category of affine noetherian $k$-schemes and that of sets. We will 
write $\HilbFunctorR{p(t)}{n}{r'}(A)$ for $\HilbFunctorR{p(t)}{n}{r'}(\Spec(A))$, for every $k$-algebra $A$.

Once proved that, for every $r'$, $\HilbFunctorR{p(t)}{n}{r'}$ is representable,   we will get as a consequence that  $\HilbFunctor{p(t)}{n}$  
is representable and that for every $r'<r$, $\HilbSchemeR{p(t)}{n}{r'}$ is an open subscheme of $\HilbScheme{p(t)}{n}$.
\end{remark}

\section{Embeddings in  Grassmann functors}

Following the classical setting \citep{ACG,CS,gotz, Gro,ik}, we will construct the scheme representing the Hilbert functor with bounded regularity in a 
suitable Grassmannian. Hence, we briefly recall the Grassmann functors (see  \citep[pp.110--114]{nit}).
 
Let us fix integers $p,q,N$ such that $0 < p < N$ and $q=N-p$.  The Grassmann functor is the functor
$\GrassFunctor{p}{N}: \Kalg \rightarrow \Sets$ that associates to any $k$-algebra $A$ the set
\begin{multline} \label{liberi}
\GrassFunctor{p}{N}(A) = \left\{\begin{array}{l} \text{isomorphism classes of locally free}\\ \text{quotient } A^N 
\rightarrow Q \text{ of rank } p \end{array}\right\} = 
 \\ = \left\{W \subset A^N\ \vert\ \rk(W)=q  \text{ and }A^N/W\text{ locally free } \right\} 
\end{multline}
and to any morphism $f : A \rightarrow B$
\[
\begin{split}
\GrassFunctor{p}{N}(f):\ \parbox{2cm}{\centering $\GrassFunctor{p}{N}(A)$} &\rightarrow\parbox{3cm}{\centering $\GrassFunctor{p}{N}(B)$} \\
 (A^N \rightarrow Q)\ &\mapsto\  (B^N \rightarrow Q\otimes_A B).
\end{split}
\]

The scheme representing this functor is the Grassmannian $\GrassScheme{p}{N}$. By   the \Pl\ embedding, $\GrassScheme{p}{N}$ becomes a subscheme of 
the projective space $\PP^E$, where $E=\binom{N}{p}$.

 If $X \in \HilbFunctorR{p(t)}{n}{r'} (A)$, from the cohomology long exact sequence of
\begin{equation*}\label{eq:shortExactSequence}
0 \longrightarrow \mathcal{I}_{X} \longrightarrow \mathcal{O}_{\PP^n_A} \longrightarrow \mathcal{O}_X \longrightarrow 0
\end{equation*}
 twisted by any $s\geq r'$, we can deduce the  short
 exact sequence of global sections
\begin{equation}\label{eq:esatta}
0 \longrightarrow H^0\big(\mathcal{I}_{X}(s)\big) \longrightarrow H^0 \big(\mathcal{O}_{\PP^n_A}(s)\big) \longrightarrow H^0\big(\mathcal{O}_X(s
)\big) \longrightarrow 0.
\end{equation}
 
  The  $A$-module $H^0 \big(\mathcal{O}_{\PP^n_A}(s)\big)$ is free of rank $N(s)$  and,   by flatness,  $H^0\big(\mathcal{O}_{X}(s)\big)$ is locally
  free of rank $p(s)$. Then, the $A$-module $H^0\big(\mathcal{I}_X(s
)\big)$ is locally free with rank  $q(s)= N(s) -p(s)$.  
 Using (\ref{eq:esatta}) we get  a natural transformation 
\begin{equation}\label{eq:changeCoordFucntors}
\underline{\mathscr{H}}^{[s]} \colon \HilbFunctorR{p(t)}{n}{r'} \longrightarrow \GrassFunctor{p(s)}{N(s)} 
\end{equation} 
in the following way. For for each $k$-algebra $A$ and each $X\in \HilbFunctorR{p(t)}{n}{r'}(A)$
we take
\[ \underline{\mathscr{H}}^{[s]}(A)(X)= P_s\otimes_k A /H^0(\mathcal I_X(s))=H^0\big(\mathcal{O}_{X}(s
)\big) \]
and for  each homomorphism $f: A \rightarrow B$ of $k$-algebras, set
\[
\begin{split}
\phi=\HilbFunctorR{p(t)}{n}{r'} (f):&\ \HilbFunctorR{p(t)}{n}{r'} (A)\ \rightarrow \parbox{3cm}{\centering $\HilbFunctorR{p(t)}{n}{r'} (B)$}\\
& \parbox{2.2cm}{\centering $X$} \mapsto\  X':=X \times_{\Spec A} \Spec B
\end{split}
\]
we take
\[ \underline{\mathscr{H}}^{[s]}(\phi) \colon \ H^0\big(\mathcal{O}_X(s
)\big) \mapsto  H^0\big(\mathcal{O}_{X '}(s
)\big) .\]

\section{Marked bases and schemes}\label{sec:generalSetting}

  For denoting terms, we will use the multi-index notation, i.e.,  $x^\alpha := x_0^{\alpha_0} \cdots
 x_n^{\alpha_n}$ 
for every $\alpha = (\alpha_0,\ldots,\alpha_n) \in \NN^{n+1}$. When the order of terms comes into play, we will consider the 
degree reverse lexicographic order $\RevLex$ assuming $x_0 < \ldots < x_n$. For any term $x^\alpha$,  let $\min(x^\alpha)$ denote the minimal variable
which divides $x^\alpha$. If $J$ is a monomial ideal, we will denote by $\cN(J)$ the set of terms in $P\setminus J$. For a subset $V$ of a usual
graded ring $R=\bigoplus_{t}R_t$,   $V_s$ and $V_{\geq s}$  will denote respectively $V\cap R_s$ and $V\cap \bigoplus_{t\geq s}R_t$. 

An ideal $I \subset P$ is said \emph{Borel-fixed} (Borel for short) if it is fixed by the action of the Borel subgroup, i.e. by the
 group of the
 upper  triangular matrices. These ideals are well studied, mainly because of two reasons:
\begin{itemize}
\item \citet{Galligo}  and \citet{BS} proved that the generic initial ideal of any ideal is 
Borel-fixed, 
which means, in the context of Hilbert schemes, that any component and any intersection of components of $\HilbScheme{p(t)}{n}$
 contains at 
least a point corresponding to a scheme defined by a Borel-fixed ideal;
\item in characteristic zero, an ideal $J$ is Borel-fixed if, and only if, it is a strongly stable  monomial ideal, i.e.,  it is generated by terms,  and 
for each term $x^\alpha \in J$, then also the term $\frac{x_j}{x_i} x^\alpha$
is 
in  $J$ for all $x_i \mid x^\alpha$ and $x_j > x_i$ (we recall that we are assuming $x_0 < \cdots < x_n$).
\end{itemize}

For these reasons, from now on we assume that $\mathrm{char}(k)=0$.

Our proofs heavily use the theory of marked bases (\citet{CR,BCLR,LRFunt}). We recall some of the results and notation contained
in the quoted papers.

Let $A$ be a $k$-algebra. Set $T:=P\otimes_k A$. Obviously, the set of terms of any degree $s$ of a Borel-fixed ideal generates a Borel-fixed ideal. For each $f\in T$,  
the support 
$\supp(f)$ of $f$ is the set of terms  that appear in $f$ with  non-zero coefficient.

\begin{definition}{\citep[Definitions 1.3 and 1.4]{CR}}\label{basemarcata}
A \emph{monic marked polynomial} (marked polynomial for short) is a polynomial $f\in T$ together with a specified term
$x^\alpha$ of $\supp(f)$,  called \emph{head term}
of $f$ and denoted by $\Ht(f)$.  We assume furthermore that the coefficient of $x^\alpha$ in $f$ is $1_A$. Hence, we can write a marked polynomial 
as $f_\alpha = 
x^\alpha-\sum c_{\alpha\gamma}x^\gamma$, with $x^\alpha =\Ht (f_\alpha)$, $x^\gamma \neq x^\alpha$ and $c_{\alpha\gamma}\in A$. 

A finite set $F$ of homogeneous marked polynomials $f_\alpha = x^\alpha-\sum c_{\alpha\gamma}x^\gamma$, with $Ht(f_\alpha) =x^\alpha$, 
is called a \emph{$J$-marked set} if the head terms $x^\alpha$ form the  minimal monomial basis $B_J$ of the monomial ideal $J$ as a $T$-module
 and every $x^\gamma$ is an element of  $\cN(J)$. Hence,  $\vert\supp(f_\alpha) \cap J\vert = 1$.

 A $J$-marked set $F$ is a \emph{$J$-marked basis} if  the quotient $T/(F)$ is freely generated by $\mathcal N(J)$ as an $A$-module.
\end{definition}
\begin{lemma}\label{piatto}\citep[Section 3]{LRFunt}
Let $J$ be a monomial ideal and $I$ be a homogeneous ideal in $T$ generated by a $J$-marked basis. Then $\Proj(T/I)$ is $A$-flat. 
\end{lemma}
\begin{proof}
It is sufficient to prove that $T/I$ is an $A$-flat module \citep[Section II, Proposition 5.2 and Section III, proof of Proposition 9.2]{AG}. By definition 
of marked basis, $T/I$ is freely generated by $\mathcal N(J)$ as an $A$-module, hence it is $A$-flat.
\end{proof}

Let $J$ be a strongly stable ideal  and $I$ be the ideal in $T$ generated by    a $J$-marked set $F$.  
 For every  $s$,     let us denote by
$F^{(s)}$   the following set of polynomials marked on  $(J_{s})$: 
\[F^{(s)} :=\{ x^\eta f_\alpha \ \vert \  f_\alpha \in F,\ deg(x^{\eta}f_{\alpha})=s,\ \Ht(x^\eta f_\alpha)=x^\eta x^\alpha, \  \max(x^\eta ) \leq \min(x^\alpha)  \}. \]

\begin{theorem}\label{th:rifatto}\citep[Lemma 2.7 (iv) and Theorem 2.9]{LRFunt} Let $J$ be a strongly stable ideal generated in degree $m$, 
$p(t)$ be the Hilbert polynomial of $\mathrm{Proj}(T/J)$, $q(t)$ be the polynomial $N(t)-p(t)$, 
and $I$ be the ideal in $T$ generated by 
a $J$-marked set $F$.  

Then,  for every $t\geq m$, $ F^{(t)} $ generates a free $A$-module   of rank $q(t)=\rk ( J_t)$ and  the free $A$-module $ \langle F^{(t)} \rangle$  contains a unique 
$J_t$-marked set $\widetilde{F}^{(t)}$.

Moreover, 
the following conditions are equivalent:
\begin{enumerate}[(i)]
\item\label{it:rifatto_i} $F$ is a $J$-marked basis.
\item\label{it:rifatto_ii}  For all $t\geq m$,  $I_t$ is a free $A$-module with basis $\widetilde{F}^{(t)}$.
\item\label{it:rifatto_iii}  For all $t\geq m$,  $I_t$ and $J_t$ are  free $A$-modules with the same rank.
\item\label{it:rifatto_iv}    $\widetilde{F}^{(m+1)}$ generates $I_{m+1}$ as an A-module.
\item\label{it:rifatto_v} $\wedge^{q(m+1)+1} I_{m+1}=0$.
\end{enumerate}
\end{theorem}

\begin{proof}
For the proof of the first statements and the equivalences of items from (\ref{it:rifatto_i}) to (\ref{it:rifatto_iv}), we refer to \cite{LRFunt}.

For the equivalence of item (\ref{it:rifatto_iv}) with (\ref{it:rifatto_v}), it is enough to consider the matrix of the coefficients of any set of generators 
of $I_{m+1}$ containing the marked set $\widetilde F^{(m+1)}$ and use Gaussian reduction to obtain a new set of generators for $I_{m+1}$.  Every element of 
this new set of generators is also an element of  the free $A$-module $\langle \widetilde F^{(m+1)}\rangle$ if, and only if, $\wedge^{q(m+1)+1} I_{m+1}=0$.
\end{proof}

 If $A=K$ is a field,  a homogeneous ideal $I$ in $T=P\otimes_k K$ is said $s$-regular if it has a  graded free resolution
\[ 0 \rightarrow \dots   \rightarrow \oplus_{j=1}^{t_i} T(-b_{i,j}) \rightarrow  \dots  \rightarrow 
\oplus_{j=1}^{t_0} T(-b_{0,j}) \rightarrow I \rightarrow 0\]
 such that $ b_{i,j} \leq s+i$  for all $b_{i,j}$.
The Castelnuovo-Mumford regularity $\reg(I)$ of $I$  is the smallest such $s$. 
The regularity of the scheme  $X=\Proj (T/I) \subset \PP^n_K$ coincides with that of the saturated ideal $I^\sat$, 
while in general $\reg(X) \leq \reg(I)$ (\citet[Proposition 2.6]{Green}).
The regularity of a strongly stable ideal $J$ is the maximal degree of terms in its minimal monomial basis \citep[Proposition 2.11]{Green}.

\begin{theorem}
 \label{cor:regularity}
Take $J=(J_m)$, and $I$ as in Theorem \ref{th:rifatto} and assume that $I$ satisfies the equivalent properties listed there.
Then $\reg(I)=\reg(J)=m$ and 
$\reg(\Proj(T/I))\leq \reg(\Proj(T/J))$.
\end{theorem}
\begin{proof}
 We expand the proof of \citet[Proposition 4.6]{CR}.  
 Take the minimal free resolutions of the ideals $I$ and $J$. 
 By hypothesis, the ideals $I$ and $J$  are  generated by $q_0:=q(m)$ homogeneous polynomials of degree $m$. 
 Hence, in their minimal resolutions, the initial free module is $T^{q_0}(-m) $ and those of the $h$-th syzygies  are  direct
 sums of modules  $T(-m-h-k_{h,j})$  for non-negative integers $k_{h,j}$. 
  Moreover $\reg(J)=m$, so that its $h$-syzygies are direct sums of modules  $T(-m-h)$.  We need to prove that the same is true for
 the $h$-th syzygies of $I$.

 Fix an an integer
 $i>0$ and suppose that the modules of $h$-th syzygies of $I$ and $J$ are equal for each $h<i$:
\[   \oplus_{j=1}^{t_i} T(-m-i-d_{i,j}) \rightarrow T^{q_{i-1}}(-m-i+1) \rightarrow \dots  \rightarrow T^{q_0}(-m) \rightarrow I \rightarrow 0\]
\[  \qquad \qquad    T^{q_i}(-m-i) \rightarrow T^{q_{i-1}}(-m-i+1) \rightarrow \dots  \rightarrow T^{q_0}(-m) \rightarrow J \rightarrow 0\]
 Now we compare the module $\Theta:=\oplus_{j=1}^{t_i} T(-m-i-d_{i,j}) $ of $i$-th syzygies of $I$ with $ T^{q_i}(-m-i) $,  the one of  $J$,
 taking in account that  the two ideals share, by hypothesis, the same Hilbert function (Theorem \ref{th:rifatto} (\ref{it:rifatto_iii})). 
Let $x$ be the number of terms in $\Theta$ with $d_{i,j} =0$. By     the equality 
$\rk (I_{m+i})=\rk(J_{m+i})$,  we obtain  $x =q_i$.
Assume for a moment that $\Theta$ has a component with $d_{i,j} >0$
and take the one with $d =d_{i,j}$ minimal. The computation of the dimensions in degree $m+i+d$ gives that
any addendum $T(-m-i-d)$ of $\Theta$ has image contained in that of the submodule $T^{q_i}(-m-i)$ of $\Theta$,
contradicting the minimality of the resolution of $I$.

Since $I$ and $J$ have minimal free resolutions with the same Betti numbers, we get $\reg (I) = \reg (J)$.

Set $t:=\reg(J^\sat)$ and let  $J'$,  $I'$ be the ideals generated by $(J^\sat)_t$ and $(I^\sat)_t$ respectively.   In \citet[Theorem 4.7]{BCLR} it is proved
that $ I'$ is generated by a $J'$-marked basis. By applying the first part of the proof to the ideals $J'$ and $I'$ we obtain $\reg(I') = \reg(J')=t=\reg(J^\sat)$. 
  Among all the homogeneous ideals with the same saturation, the saturated
ideal is the one with minimal regularity. Then  $\reg(\Proj(T/I))=\reg(I^{\sat}) 
\le \reg(I') =\reg(J^\sat)=\reg(\Proj(T/J))$.
\end{proof}

For a fixed saturated strongly stable ideal $J$, it is possible to define a functor $\underline{\mathbf{Mf}}_{J_{\geq s}}$ between the category of $k$-algebras 
and that of sets. For every $k$-algebra $A$,  $\underline{\mathbf{Mf}}_{J_{\geq s}}(A)$ is the set of ideals in $A[x_0,\dots,x_n]$ generated by a 
$J_{\geq s}$-marked basis. Such a functor is represented by a closed subscheme of $A^{D}$, for a suitable $D$ depending on $s$, $n$ and the Hilbert polynomial 
of $\Proj(T/J)$. For the main features of $\underline{\mathbf{Mf}}_{J_{\geq s}}$  and the proof of its representability see \citep{LRFunt}.

\section{Open cover of $\GrassFunctor{p(s)}{N(s)}$}
As observed in Remark \ref{zariski},  $\HilbFunctorR{p(t)}{n}{r'}$ is a Zariski sheaf and it is a subfunctor of the representable functor $\GrassFunctor{p(s)}{N(s)}$. 
Then it is sufficient to check its representability  focusing on an open
 cover of $\GrassFunctor{p(s)}{N(s)}$ \citep[Proposition 2.7, Corollary 2.8]{HaimSturm}. Following the line of 
 \citet{BLMR}, we consider a family of open subfunctors of $\GrassFunctor{p(s)}{N(s)}$, whose definition involves the action of the linear group
$\PGL(n+1)$.

 For every element $g\in \PGL(n+1):=\PGL_\QQ(n+1)$, $\widetilde{g}$ is the automorphism induced by  $g$ on $P$, $T$ or on  Grassmann and Hilbert functors, and $g\centerdot$ is the
corresponding action on an element.

We fix a basis for the $A$-module  $P_s\otimes_k A$ choosing the terms of degree $s$ ordered decreasingly with respect to $\RevLex$.
We denote by
$x^{\alpha(j)}$ the $j$-th term of this basis, with  $j=1, \dots N(s)$. In what follows we identify
 $P_s\otimes_k A$ with $A^{N(s)}$ by the chosen basis.

Every set of  $p(s)$ indices $\mathcal{I}=  (i_1,\ldots,i_{p(s)})$, $1\leq i_1 < \dots <i_{p(s)} \leq N(s)$, corresponds
to a  Pl\"ucker coordinate $\Delta_{\mathcal{I}}$ of the  Pl\"ucker embedding   of $\GrassScheme{p(s)}{N(s)}$.   We  associate to  $\mathcal{I}$ the monomial
ideal  $J(\mathcal{I})$  generated by the terms of degree $s$ whose indices  are not elements of $\mathcal I$.

\begin{notation}\label{not}
 For every $r'\leq s \leq r$ we denote by 
 \begin{itemize}
\item $\mathcal{S}^{[s]}$ the set of multi-indices  $\mathcal{I}\subset \{1, \dots, N(s)\}$, $\vert \mathcal I\vert=p(s)$, such that 
$J(\mathcal{I})$ is a Borel ideal;
\item   $\mathcal{S}^{[r',s]}_{p(t)}$ the set of $\mathcal I\in\mathcal{S}^{[s]}$ such that $\Proj(P/J(\mathcal I))$
has Hilbert polynomial $p(t)$
and   is $r'$-regular;   
\item    $\mathbb{B}^{[r']}_{p(t)}$ the set of saturated Borel ideals defining points of
$\HilbSchemeR{p(t)}{n}{r'}$;
\item  ${L}^{[r',s]}_{p(t)}$ the closed subset of $\GrassScheme{p(s)}{N(s)}$ defined by the ideal 
$$\left( g\centerdot  \Delta_{\mathcal{I}} \ \vert \ \forall \  g \in \PGL(n+1), \ \forall \
 \mathcal{I} \in \mathcal{S}^{[r',s]}_{p(t)} \right).$$ 
  \end{itemize}
\end{notation}

\begin{lemma}\label{corrispondenza}
There is a natural 1:1 correspondence
 \[ \alpha^{[r',s]}\colon \mathbb{B}^{[r']}_{p(t)} \rightarrow \mathcal{S}^{[r',s]}_{p(t)}\]
 given by $\alpha^{[r',s]}(J) =\mathcal{I}$ where 
 $J(\mathcal{I})^{\sat}=J$.
\end{lemma}
\begin{proof}
 Every $J \in \mathbb{B}^{[r']}_{p(t)}$  contains exactly $q(s)=N(s)-p(s)$ terms of degree $s$, for every $s\geq r'$.
 Indeed the assumption on the regularity implies that $\rk (J_s)=q(s)$. 
 Hence, for every $s\geq r'$, $J\in \mathbb{B}^{[r']}_{p(t)}$ 
 identifies a set of $p(s)$  indices $\mathcal{I}\subset \{ 1, \dots ,N(s)\}$ such that $J=J(\mathcal{I})^\sat$. 
Viceversa, for every  $\mathcal{I} \in \mathcal{S}^{[r',s]}_{p(t)}$, the  saturated  ideal $J(\mathcal{I})^{\sat}$  is Borel fixed, because
 $J(\mathcal{I})$ is,  and taking saturation preserves this property. Moreover, $J(\mathcal{I})^{\sat}\in \mathbb{B}^{[r']}_{p(t)}$, because 
 it  is the saturated  ideal defining  $\Proj( P /J(\mathcal I))$,  that has Hilbert polynomial $p(t)$ and is  $r'$-regular.
\end{proof}

\medskip

Let $(\b_j, j=1, \dots, p(s))$ denote the standard basis of $A^{p(s)}$. For every  $\mathcal{I}=  (i_1,\ldots,i_{p(s)})$,   and $g \in \PGL(n+1)$, we consider 
the injective morphism
\[
\begin{split}
\Gamma^{[s]}_{\mathcal{I}}: A^{p(s)} &\rightarrow A^{N(s)} \\
\b_j & \mapsto x^{\alpha(i_j)}
\end{split}
\]

and the subfunctor $\underline{\mathbf{G}}^{[s]}_{\mathcal{I},g}$ of $\GrassFunctor{p(s)}{N(s)}$ given by:
\begin{equation}\label{liberi2}
\underline{\mathbf{G}}^{[s]}_{\mathcal{I},g}(A) = \left\{ \begin{array}{c} \text{locally free quotient }
A^{N(s)} \stackrel{\pi_W}
{\longrightarrow} A^{N(s)}/W\text{ of rank } p(s) \\ \text{such that } \pi_W \circ \widetilde{g}\circ 
\Gamma^{[s]}_{\mathcal{I}}
 \text{ is surjective}\end{array} \right\}.
\end{equation}

\begin{remark}\label{modlibero} The elements of $\underline{\mathbf{G}}^{[s]}_{\mathcal{I},g}(A)$ correspond in a natural way to ideals in $P\otimes_k A$
 having a special set of generators. Let us fix by simplicity $g=Id$. 
The quotient  $A^{N(s)}/W$ appearing in \eqref{liberi2} is actually a free module of rank $p(s)$,  since the map $\pi_W\circ \Gamma_{\mathcal I}^{[s]}$ is
defined between projective modules of the
same finite rank and it is surjective.
 Hence, also the submodule $W$ is  free of rank $q(s)$, and $A^{N(s)}\simeq A^{N(s)}/W\oplus W$.

 If we identify  $A^{N(s)}/W$  
with  $A^{p(s)}=\left\langle x^{\alpha(i)} , i\in \mathcal I\right\rangle$   via $\pi_W \circ\Gamma^{[s]}_{\mathcal{I}}$, we reduce the problem to 
study surjective morphisms $A^{N(s)} \rightarrow A^{p(s)}$.   These maps 
are characterized 
by their kernel or, equivalently, by the images in $A^{p(s)}$ of the elements  $x^{\alpha(i)}$ such that    $i\notin \mathcal I$. In this way we 
obtain a set of generators for the kernel $W$ of the type $f_{\alpha(j)}=x^{\alpha(j)}-\sum c_{ji} x^{\alpha(i)}$ with $j\notin \mathcal I$,
$i \in \mathcal I$, and $c_{ji}\in A$. 

Therefore,  in this setting, it is natural to identify the elements of $\underline{\mathbf{G}}^{[s]}_{\mathcal{I},g}(A)$ with  the ideals in  
$P\otimes_k A$ generated by   $J(\mathcal I)$-marked sets. 
\end{remark}

\begin{prop}\label{pr:BorelCoverG}
The open subfunctor $\underline{\mathbf{G}}^{[s]}_{\mathcal{I},g}$ is represented by the open subscheme  ${\mathbf{G}}^{[s]}_{\mathcal{I},g}$ of 
$\GrassScheme{p(s)}{N(s)}$ given by the non-vanishing of  $g\centerdot \Delta_{\mathcal I}$.

The scheme ${\mathbf{G}}^{[s]}_{\mathcal{I},Id}$ is naturally isomorphic to $\mathrm{Spec}(k[C])$, where $C$ is the  set of variables
$\lbrace C_{ji} \vert \, j \notin \mathcal I, \ i\in \mathcal I\rbrace$. 
\end{prop}
\begin{proof} In \citet[Lemma 4.2]{BLMR} the authors prove that the collection of subfunctors \eqref{eq:GrBorelSubFunctors} 
covers the Grassmann functor $\GrassFunctor{p(s)}{N(s)}$.  The  open subfunctor $\underline{\mathbf{G}}^{[s]}_{\mathcal{I},Id}$ is represented by 
the open subscheme of $\GrassScheme{p(s)}{N(s)}$ defined by 
the non-vanishing of the Pl\"ucker coordinate $\Delta_{\mathcal I}$ \citep[Section 2]{BLMR}. As a consequence,  we get the claim 
 for
$\underline{\mathbf{G}}^{[s]}_{\mathcal{I},g}$. 

For the second 
statement, we refer to \cite[Section 2, in particular Formula 2.3]{BLMR}. 
The scheme isomorphism between ${\mathbf{G}}^{[s]}_{\mathcal{I},Id}$ and $\mathrm{Spec}(k[C])$ is the natural one given by local coordinates on the 
Grassmannian \citep[see Formula 2.3]{BLMR}.  Notice that, for every  $k$-algebra $A$, the $J(\mathcal I)$-marked set
corresponding to an element of $\underline{\mathbf{G}}^{[s]}_{\mathcal{I},Id}(A)$ presented in Remark \ref{modlibero} is in fact 
determined by the list of the coefficients $c_{ji}\in A$, i.e.,  by a $k$-algebras morphism $k[C]\rightarrow A$.  
\end{proof}

\begin{corollary}\label{ricpiccoloG}
The collection of open subfunctors
\begin{equation}\label{eq:GrBorelSubFunctors}
\left\{\underline{\mathbf{G}}^{[s]}_{\mathcal{I},g} \quad \Big\vert\quad \begin{array}{l}\forall\ g
\in \PGL(n+1)\text{ and } \forall\
 \mathcal{I} \in \mathcal{S}^{[s]} \end{array}\right\}
\end{equation}
covers the Grassmann functor $\GrassFunctor{p(s)}{N(s)}$ and the collection of the corresponding open subschemes ${\mathbf{G}}^{[s]}_{\mathcal{I},g}$
covers $\GrassScheme{p(s)}{N(s)}$.

The collection of   open subschemes
\begin{equation*}
\left\{{\mathbf{G}}^{[s]}_{\mathcal{I},g}\quad \Big\vert\quad
\forall\ g
\in \PGL(n+1) \text{ and }\forall\ \mathcal{I} \in \mathcal{S}^{[r', s]}_{p(t)}\right\}
\end{equation*}
covers $\GrassScheme{p(s)}{N(s)}\setminus {L}^{[r',s]}_{p(t)}$.
\end{corollary}

\begin{proof}
Both statements follow from Proposition \ref{pr:BorelCoverG}, remembering the definitions of $\mathcal {S}^{[r', s]}_{p(t)}$ and ${L}^{[r',s]}_{p(t)}$ 
given in Notations \ref{not}.
\end{proof}

\section{Open cover of  $\HilbFunctorR{p(t)}{n}{r'}$ and representability}

   In order to prove that  $\HilbFunctorR{p(t)}{n}{r'}$ is representable and embed its representing scheme  as a locally closed subscheme in
   $\GrassScheme{p(s)}{N(s)}$, we follow the lines of \citet{BLMR}.
In that paper, the authors characterize the elements in $\HilbScheme{p(t)}{n}$  in  $\GrassScheme{p(r)}{N(r)}$ using Macaulay's 
estimate of growth and Gotzmann's persistence theorem (see \citet{Green}, \citet{gotz}).
These two results allow a criterion to establish whether an element of $ \GrassScheme{p(r)}{N(r)}$ is also an element of  $\HilbScheme{p(t)}{n}$.  
Indeed, $ P/I$ has Hilbert polynomial $p(t)$ if, and only if,  the rank of the  module   $I_{r+1}$ does not exceed $q(r+1)=N(r+1)-p(r+1)$, namely 
$\wedge^{q(r+1)+1}I_{r+1}$ vanishes.  

However, these two results do not hold if $s<r$, as shown in the following example.

\begin{example}\label{ex:nonGot}
We consider the Hilbert polynomial  $p(t)=2t+2$, whose Gotzmann number is $r=3$. We fix $r'=s=2$. 
The set $\HilbFunctorR{2t+2}{3}{2}(k)$ is non-empty: for instance, it contains the point defined by  $J=(x_3^2, x_3x_2,x_2^2,x_3x_1)$.

We embed $\HilbFunctorR{2t+2}{3}{2}$ in $\GrassFunctor{6}{10}$ and we consider  $I=(x_3^2, x_3x_2, x_3x_1,x_3x_0)\in\GrassFunctor{6}{10}(k)$: $I$ 
is generated by $q(2)=4$ terms and
$\rk(I_3)=10 <q(3)=12$. Hence
 the conditions $\rk (I_{r'})=q(r')$,  $\rk (I_{r'+1})\leq q(r'+1)$ hold true, though 
$I\notin \HilbFunctorR{2t+2}{3}{2}(k)$: indeed, $\Proj(P/I)$ has Hilbert polynomial $\frac{t^2+3t+2}{2}\neq p(t)$.
\end{example}
Instead of Macaulay's estimate of growth and Gotzmann's persistence theorem, we will use the properties of marked bases given in 
Theorem \ref{th:rifatto}, in particular the equivalence between (\ref{it:rifatto_ii}), (\ref{it:rifatto_iii}) and (\ref{it:rifatto_v}).

\begin{definition}\label{famiglia} 
For every $g \in \PGL(n+1)$, $s\geq r'$ and $
 \mathcal{I} \in \mathcal{S}^{[s]}$, 
 we will denote by $\underline{\mathbf{H}}^{[r',s]}_{\mathcal{I},g} $ the open subfunctor of $\HilbFunctorR{p(t)}{n}{r'}$:
\[\left(\underline{\mathscr{H}}^{[s]}\right)^{-1}\left(\underline{\mathbf{G}}^{[s]}_{\mathcal{I},g}\right) \cap\HilbFunctorR{p(t)}{n}{r'}.\]
\end{definition}

 By Corollary \ref{ricpiccoloG}, it immediately follows that, for every $s\geq r'$,  the  collection of open subfunctors 
$\underline{\mathbf{H}}^{[r',s]}_{\mathcal{I},g} $  covers of 
  $\HilbFunctorR{p(t)}{n}{r'}$. In general, this collection  depends on $s$. However,  for every $s\geq r'$ we will extract  a  suitable subcollection
that it is still an open cover of  $\HilbFunctorR{p(t)}{n}{r'}$ and does not depend on $s$.  In particular, for $s=r'=r$, the open subfunctors 
$\underline{\mathbf{H}}^{[r,r]}_{\mathcal{I},g} $ are exactly the open cover of $\HilbFunctor{p(t)}{n}=\HilbFunctorR{p(t)}{n}{r}$ defined in \citet{BLMR}.

First of all, we observe that some of the subfunctors $\underline{\mathbf{H}}^{[r',s]}_{\mathcal{I},g} $  can be empty.  
For instance, $\underline{\mathbf{H}}^{[r,r]}_{\mathcal{I},\textit{Id}} (A)=\emptyset$ 
 if  the Hilbert function of $\Proj(P/J(\mathcal{I}))$ is lower than $p(t)$ in degrees $t\gg 0$. In fact, in this case the ideal 
 corresponding to any  point of $\underline{\mathbf{G}}^{[r]}_{\mathcal{I},Id}(A)$ has Hilbert function lower than $p(t)$ for $t\gg0$ (Theorem \ref{th:rifatto}).

\begin{example}\label{ex:nonGot2}
Let us consider  $p(t)=2t+1$, the Hilbert polynomial of the conics, whose  Gotzmann number is $r=2$. Hence, $\HilbFunctorR{2t+1}{3}{2}=\HilbFunctor{2t+1}{3}$ 
is non-empty 
and  can be embedded in $\GrassFunctor{5}{10}$. 
However, for the set of indices $\mathcal{I}=\{6,7,8,9,10\}$, the open subfunctor  $\underline{\mathbf{H}}^{[2,2]}_{\mathcal{I},\mathit{Id}}$ is empty, since 
$J(\mathcal I)=(x_3^2, x_3x_2,x_2^2, x_3x_1,x_2x_1)$ is a Borel ideal with 
 Hilbert polynomial $t+3$, lower than $ p(t)=2t+1$ for every $t>2$. 
\end{example}

 Moreover, for any two fixed integers $s,s'$ such that  $r'\leq s' <s \leq r$, there is not a canonical  way to associate to every 
$\underline{\mathbf{H}}^{[r',s]}_{\mathcal{I},g} $ 
an open subfunctor $\underline{\mathbf{H}}^{[r',s']}_{\mathcal{I'},g} $.  By Lemma \ref{corrispondenza}, we obtain such a 1:1 correspondence,  if we 
consider the
subcollection of the $\underline{\mathbf{H}}^{[r',s]}_{\mathcal{I},g} $  with $\mathcal I \in \mathcal{S}^{[r', s]}_{p(t)}$ and the  bijection $\alpha^{[r',s]}$ 
given in 
Lemma \ref{corrispondenza}. 

\begin{remark}\label{uguale} Let $A$ be a $k$-algebra and fix any $\mathcal{I} \in \mathcal{S}^{[r', s]}_{p(t)}$. 
In Remark \ref{modlibero} we observe that we may  think of $\underline{\mathbf{G}}_{\mathcal{I},Id}^{[s]}(A)$ as to 
the set of ideals generated by  $J(\mathcal I)$-marked sets. Among them, the ideals generated by 
$J(\mathcal I)$-marked bases are precisely those 
 with  Hilbert polynomial $p(t)$, i.e.,  those defining elements of $\HilbFunctor{p(t)}{n}(A)$.  
Moreover,  for every    such ideal $I$, the regularity of $\Proj(P\otimes_k A/I)$  is bounded by  the regularity  of
$\Proj(P\otimes_k A/(J(\mathcal I)))$, which is at most $r'$, by Theorem \ref{cor:regularity}.  
Furthermore, $\Proj(P\otimes_k A/I)$ is $A$-flat as shown in Lemma \ref{piatto}.
\end{remark}

With an abuse of notation,  we sometimes identify the functor  $\underline{\mathbf{H}}^{[r',s]}_{\mathcal{I},g} $  with 
its isomorphic image in   $\underline{\mathbf{G}}^{[s]}_{\mathcal{I},g}$ or with the $J(\mathcal I)$-marked functor
$\MFFunctor{J(\mathcal I)}$,  defined in  \citet[Section 3]{LRFunt} developing the concept of marked basis.

\begin{theorem}\label{lem:HilbRBorelSubFunctors} 
For every $s\geq r'$, the collection of open subfunctors
\begin{equation}\label{eq:HBorelSubFunctors}
\left\{\underline{\mathbf{H}}_{\mathcal{I},g}^{[r',s]} \quad \Big\vert\quad \begin{array}{l}\forall\ g \in \PGL(n+1)\text{ and } \forall\
 \mathcal{I} \in \mathcal{S}^{[r',s]}_{p(t)} \end{array}\right\}
\end{equation}
covers $\HilbFunctorR{p(t)}{n}{r'}$ and does not depend on $s$.

Moreover, for every  Borel ideal   $J$ in $ \mathbb{B}^{[r']}_{p(t)}$,  if
 $\mathcal I =\alpha^{[r',s]}(J)$, and  $\mathcal K =\alpha^{[r',r]}(J)$ 
 as in Lemma \ref{corrispondenza}, then 
\begin{enumerate}[(i)]
\item \label{lem:HilbRBorelSubFunctors_i} 
$\underline{\mathbf{H}}^{[r,r]}_{\mathcal{K},g}$, $ \underline{\mathbf{H}}^{[r',r]}_{\mathcal{K},g}$ and $ \underline{\mathbf{H}}^{[r',s]}_{\mathcal{I},g}$ 
are equal as subfunctors of $\HilbFunctor{p(t)}{n}$.
\item \label{lem:HilbRBorelSubFunctors_ii}   $\underline{\mathbf{H}}^{[r',s]}_{\mathcal{I},g}$  is  a
closed subfunctor of $\underline{\mathbf{G}}^{[s]}_{\mathcal{I},g} $ and it is the functor of points of a closed subscheme ${\mathbf{H}}^{[r',s]}_{\mathcal{I},g}$ of 
$\mathbf{G}^{[s]}_{\mathcal{I},g}$.
\end{enumerate}
\end{theorem} 

\begin{proof} 
We obtain the first statement arguing as in  \citet[Proposition 4.9]{BLMR}. It is sufficient to prove that  every  element of 
$\HilbFunctorR{p(t)}{n}{r'}(A)$ is also an element of  $\underline{\mathbf{H}}^{[r',s]}_{\mathcal{I},g}(A)$, for some $g \in \PGL(n+1)${ and } some 
$\mathcal{I} \in \mathcal{S}^{[r',s]}_{p(t)}$, assuming that $A=K$ is a field. 

Let $I=(I_s)$ be the ideal corresponding to such an element. For a general $g\in \PGL({n+1})$, the initial ideal of $g^{-1}\centerdot I$ 
w.r.t.  $\texttt{DegRevLex}$ is a Borel ideal $J$ generated in degree $s$. Furthermore,   $\reg(\Proj \left( P\otimes_k K/J \right) )=
\reg( \Proj \left(P\otimes_k K/I\right))\leq r' $ (\citet{BS87}, see also \citet[Corollary 20.21]{e2}). Hence, $J^{\sat} \in\mathbb{B}^{[r']}_{p(t)}$ and 
$\mathcal{I}:=\alpha^{[r',s]}(J^{\sat})\in \mathcal{S}^{[r',s]}_{p(t)}$.  

Moreover, the \Gr\ basis of $g^{-1}\centerdot I$ is a $J$-marked basis, so that 
${g}^{-1} \centerdot  I\in \MFFunctor{J}(K)=\underline{\mathbf{H}}_{\mathcal{I},\mathit{Id}}^{[r',s]}(K)$, hence, 
$  I\in \underline{\mathbf{H}}_{\mathcal{I},\mathit{g}}^{[r',s]}(K)$.

The independence of this  open cover from  $s$ will follow once proved (\ref{lem:HilbRBorelSubFunctors_i}).

 It is sufficient to prove statements (\ref{lem:HilbRBorelSubFunctors_i}) and  (\ref{lem:HilbRBorelSubFunctors_ii}) in the case $g=Id$.
\begin{enumerate}[(i)]
\item   The equality $\underline{\mathbf{H}}^{[r,r]}_{\mathcal{K},Id}=\underline{\mathbf{H}}^{[r',r]}_{\mathcal{K},Id}$ 
is a
consequence of
Remark \ref{uguale}  and Theorem \ref{cor:regularity}: indeed,  for every $k$-algebra $A$ and   $I\in \underline{\mathbf{H}}^{[r,r]}_{\mathcal{K},Id}(A)$,
the scheme $\Proj(P\otimes_kA/I)$ is  $ r'$-regular, hence $I\in \underline{\mathbf{H}}^{[r',r]}_{\mathcal{K},Id}(A)$.

We get the other equality using again Remark \ref{uguale} and Theorem \ref{cor:regularity}. Indeed, 
$\underline{\mathbf{H}}^{[r',r]}_{\mathcal{K},Id}=\MFFunctor{J_{\geq r}}$, 
$\underline{\mathbf{H}}^{[r',s]}_{\mathcal{I},Id}=\MFFunctor{J_{\geq s}}$ and 
the two marked functors coincide as subfunctors of 
$\HilbFunctor{p(t)}{n}$ (see \citet[Theorem 4.7]{LRFunt};
notice that the hypothesis $J\in \mathbb{B}^{[r']}_{p(t)}$ is  essential for this result).

\item Due to what just proved, it is sufficient to show that $\underline{\mathbf{H}}^{[r',s]}_{\mathcal{I},Id}=\MFFunctor{J_{\geq s}}$  
is a closed subfunctor of $ \underline{\mathbf{G}}^{[s]}_{\mathcal{I},Id}$. We can adapt to the case of $k$-algebras the proof
of this same fact given in \citet[Theorem 4.5]{LRFunt}  
 in the more general case of   the
$\mathbb Z$-algebras. 

Let $A$ be a $k$-algebra and consider
$I\in \underline{\mathbf{G}}^{[s]}_{\mathcal{I},Id}(A)$. As shown above in Remark \ref{modlibero} and 
Proposition \ref{pr:BorelCoverG}, 
$I$ is generated  by a $J(\mathcal I)=J_{\geq s}$-marked set 
$F=\{f_{\alpha(j)}=x^{\alpha(j)}-\sum c_{ji} x^{\alpha(i)} \ \vert\ j \notin \mathcal{I},\ i\in \mathcal I\}$ and, by hypothesis,
$\mathrm{Proj}(P\otimes_k A/J)$ 
has Hilbert polynomial $p(t)$
and regularity $\leq  r'\leq s$. 
Under these conditions,  Theorem \ref{th:rifatto} states that   $I\in \MFFunctor{J_{\geq s}}(A)$   if, 
and only if,  $\wedge^{q(s+1)+1}I_{s+1}$ vanishes. Considering the set $\{x_i f\ \vert\  f\in F^{(s)},\ i=0, \dots, n\}$ 
which generates $I_{s+1}$, we obtain closed conditions given by the vanishing of a suitable set of polynomial expression in the
coefficients  $c_{ji}$'s appearing in $F^{(s)}$. In  \citet[Theorem 4.5]{LRFunt} an explicit set of such polynomials is exhibited.
\end{enumerate}
\end{proof}

We are now ready to prove the first main result of the paper.

\begin{proof1} 
By Lemma \ref{zariski}, in order to prove that $\HilbFunctorR{p(t)}{n}{r'}$ is representable, we use \citep[Proposition 2.7, Corollary 2.8]{HaimSturm}. 
We consider the open cover of the representable functor  $\GrassFunctor{p(s)}{N(s)}$ given in Corollary \ref{ricpiccoloG}.
Then, it is sufficient to observe that the subfunctor $\underline{\mathbf{H}}_{\mathcal{I},g}^{[r',s]}$, for $g \in \PGL(n+1)$ and 
$\mathcal{I} \in \mathcal{S}^{[r',s]}_{p(t)}$, is representable and is a closed subfunctor of $\underline{\mathbf{G}}_{\mathcal{I},g}^{[s]}$, 
by Theorem \ref{lem:HilbRBorelSubFunctors}, (\ref{lem:HilbRBorelSubFunctors_ii}). Hence,  $\HilbFunctorR{p(t)}{n}{r'}$ is represented by a 
closed subscheme of $\GrassScheme{p(s)}{N(s)}\setminus L^{[r',s]}_{p(t)}$. 

We now fix $r'=s=r$. By the first part of the proof, $\HilbFunctor{p(t)}{n}=\HilbFunctorR{p(t)}{n}{r}$ is represented by a locally closed subscheme of $\GrassScheme{p(r)}{N(r)}$   and furthermore, it does not intersect $L^{[r,r]}_{p(t)}$, thanks to Theorem \ref{lem:HilbRBorelSubFunctors}. The scheme $\HilbScheme{p(t)}{n}$ is closed in $\GrassScheme{p(r)}{N(r)}$, because
the functor $\HilbFunctor{p(t)}{n}$ satisfies the valuative criterion for properness (see \citet[Lemme 3.7]{Gro}, \citet[III, Proposition 9.8 and Remark 9.8.1]{AG}, \citet[5.5.7]{nit}).

The last statement follows from the previous one, as observed in Remark \ref{zariski}.
\end{proof1}

  \begin{remark}  \citet{BLMR} also prove that    $\HilbFunctor{p(t)}{n}$ is represented by a closed subscheme of $\GrassScheme{p(r)}{N(r)}$, without any flattening stratification argument.
\end{remark}

\section{Equations}\label{sec:equations}

In this section we consider the closed embedding of $\HilbSchemeR{p(t)}{n}{r'}$ in  $\GrassScheme{p(s)}{N(s)}\setminus L^{[r',s]}_{p(t)}$ studied in 
the previous section and 
exhibit global equations in the Pl\"ucker coordinates  defining it.   Our proof generalizes the one given for $\HilbScheme{p(t)}{n}$ by \citet{BLMR}. 
We follow the same lines and use the tools therein developed.

We set $d=\deg(p(t))$ and fix any integers $t\geq s\geq r'$.  Let  $A$ be a  $k$-algebra, $\mathcal{I}\in \mathcal{S}^{[r',s]}_{p(t)}$,  $J:=J(\mathcal{I})$,  
and consider any ideal $I \in 
 \underline{\mathbf{G}}^{[s]}_{\mathcal{I},Id}(A)$.  Let us denote by  $q'(t)$ and $q''(t)$ the following integers:
\[
\begin{array}{lcl}
q'(t)&=&\rk_A A[x_{d+1},\ldots, x_{n}]_t,  \\
q''(t)&=& q(t) -q'(t).
\end{array}
\]

 In this setting,
$J$ contains all the terms in ${A}[x_{d+1}, \dots, x_n]_t$ (see \citep[Lemma 1.4]{BLR}).  
Moreover, $I_t/I_t\cap (x_0, \dots,x_d)$ is a free module of rank $q'(t)$; indeed, looking at the subset $\widetilde F^{(t)}$ of $I_t$ defined 
in Theorem \ref{th:rifatto}, it is obvious that
the canonical map  $I/I\cap (x_0, \dots,x_d)\rightarrow A[x_{0}, \dots, x_n]/(x_0, \dots,x_d)=A[x_{d+1}, \dots, x_n]$ is an isomorphism of
$A$-modules in every degree $t\geq s$.

We can split $I_t$ in a direct sum $I_t=I'_t \oplus I''_t$, where  $I''_t:= I_t \cap (x_0, \dots, x_d)$ and $I'_t$ is
 any complementary  submodule. For $t=s$,  note that  $I'_s$  is free of rank $q'(s)$ and $I''_s$ is  free of rank $q''(s)$, since 
 $\mathcal I\in  \underline{\mathbf{G}}^{[s]}_{\mathcal{I},{Id}}(A)$.

Moreover, 
$I''_{t}$   is the sum of  the following  two submodules: 
\[
\begin{array}{rl} 
I^{(1)}_{t}&:=\langle x_hI_{t-1} \ \vert \ \forall h=0, \dots, d  \rangle,\\
I^{(2)}_{t}&:=\langle x_hI_{t-1}\ \vert \ \forall h=d+1,\dots, n  \rangle \cap (x_0, \dots, x_d). 
\end{array}
\]

 We underline that, for a given $I \in 
 \underline{\mathbf{G}}^{[s]}_{\mathcal{I},{Id}}(A)$,  the  submodules $I''_t$, $I^{(1)}_{t}$, $I^{(2)}_{t}$, 
do not depend on the set of indices $\mathcal{I}$, but only on the fact that $J(\mathcal{I})$ is a Borel ideal with Hilbert polynomial $p(t)$ of degree $d$, 
so that 
$J(\mathcal{I})$ contains all the terms in ${A}[x_{d+1}, \dots, x_n]_t$, for every $t\geq s$. 

Let $\widetilde F^{(t)}$ be the $J(\mathcal I)_{t}$-marked set contained in $I_{t}$ (see Theorem \ref{th:rifatto}).

\begin{lemma}\label{semplificazione}
In the above setting, for every $t\geq s$,  $I^{(1)}_t$ contains the free $A$-module $F''_t$ of rank $q''(t)$ generated by 
$\lbrace f\in\widetilde F^{(t)}\vert\, f\in(x_0,\dots,x_d)\rbrace$, while the complementary set
$\lbrace f\in\widetilde F^{(t)}\vert\, f\notin(x_0,\dots,x_d)\rbrace$ generates a free $A$-module $F'_t$, that 
we can take as $I'_t$.

 Moreover, $I_t=\langle\widetilde F^{(t)}\rangle$ if, and only if, $\wedge^{q''(t)+1}I^{(1)}_t=0$ and $\wedge^{q''(t)}I^{(1)}_t\wedge I''_t=0$.
\end{lemma}
\begin{proof}
The first statement is a consequence of the shape of the polynomials in the marked set $\widetilde F^{(t)}$ and the fact that $J=J(\mathcal I)$ is Borel 
fixed with Hilbert polynomial $p(t)$.

For the second statement, we use Theorem \ref{th:rifatto}, in particular the equivalence between (\ref{it:rifatto_ii}) and (\ref{it:rifatto_v})
giving that $I_t=\langle\widetilde F^{(t)}\rangle$ if, and only if, 
 $\wedge^{q(t)+1}I_t=0$. We can decompose $A[x_0, \dots,x_n]_t$ in the   direct sum 
 $$A[x_{d+1}, \dots,x_n]_t\oplus (x_0, \dots,x_d)_t\cap J \oplus \langle x^\alpha \in  (x_0, \dots,x_d)_t\setminus J\rangle   .$$
 Let $\pi_1$ and $\pi_2$ be   the projections of $A[x_0, \dots,x_n]_t$ to the first and to the second summand. By definition,  $\pi_1(I''_t)=0$, hence 
 $\pi_1(I'_t)=\pi_1(I_t)=A[x_{d+1}, \dots,x_n]_t$, where the second equality is obvious looking at the submodule  $F'_t$.  Indeed, for every $t\geq s$,  $J_t$ 
 contains all the terms in ${A}[x_{d+1}, \dots, x_n]_t$,   and for every $x^\alpha \in {A}[x_{d+1}, \dots, x_n]_t$ there is $f_\alpha \in F'_t\cap \widetilde F^{(t)}$ 
 such that $\Ht(f_\alpha)=x^\alpha$ and $f_\alpha-x^\alpha\in (x_0,\dots,x_n)_t$.  
 Therefore, $\wedge^{q(t)+1}I_t=0$ if, and only if, $\wedge^{q''(t)+1}I''_t=0$.
 
 Moreover, $\pi_2(I''_t)=\pi_2(I^{(1)}_t)$, since $I^{(1)}_t$   contains the submodule $F''_t$ and  $\pi_2(F''_t)=(x_0, \dots,x_d)_t\cap J $, again by the special 
 shape of the polynomials in 
 $\widetilde F^{(t)}\cap F''_t$.  
 Then we conclude that $\wedge^{q''(t)+1} I''_t=0$ is equivalent to $\wedge^{q''(t)+1}I^{(1)}_t=0$ and 
$\wedge^{q''(t)}I^{(1)}_t\wedge \langle f \rangle=0$ for every $f$ in any set of generators of $I''_t$.
\end{proof}

Recall that we fixed the basis of terms of degree $s$ for the $A$-module  $P_s\otimes_k A$  (ordered decreasingly with respect to $\RevLex$).
We denote by
$x^{\alpha(j)}$ the $j$-th term of this basis, with  $j=1, \dots N(s)$. In what follows we identify
 $P_s\otimes_k A$ with $A^{N(s)}$ by $x^{\alpha(j)}\mapsto \a_j$, where $\a_j$ is a basis of $A^{N(s)}$.

 \begin{definition}\cite[Definition 2.3]{BLMR}
Let $A$ be a $k$-algebra, $\mathcal{I} \subset (1, \dots, N)$ with  $\vert \mathcal I\vert=p(t)$  and $I  \in \underline{\mathbf{G}}^{[s]}_{\mathcal{I},Id}(A)$.
For any $1 \leq m \leq q(s)$, and for any $\mathcal{K}\subset\lbrace 1,\dots, N\rbrace$, $\vert\mathcal{K}\vert = p(s)+m$, we define
\begin{equation}
\delta^{(m)}_{\mathcal{K}}(I) := \sum_{\begin{subarray}{c} \mathcal{H} \subset \mathcal{K} \\ \vert\mathcal{H}\vert=m \end{subarray}}
\varepsilon_{\mathcal{K}\setminus\mathcal{H}} \Delta_{\mathcal{K}\setminus\mathcal{H}}(I) \a_{h_1} \wedge \cdots \wedge \a_{h_m}, 
\end{equation}
where $\varepsilon_{\mathcal{K}\setminus\mathcal{H}}\in\lbrace -1,+1\rbrace$ is the signature of a permutation of $1,\dots, N$ depending on $\mathcal K$ and $\mathcal H$ 
\citep[Lemma 2.1]{BLMR}. We also define
\[
\mathcal{B}^{m}(I):= \left\{\delta^{(m)}_{\mathcal{K}}(I)\ \vert\ \forall\ \mathcal{K} \text{ s.t. } \vert \mathcal{K} 
\vert = m+p(s) \right\},
\]
and its subset
\[
\mathcal{B}_{\mathcal{I}}^{m}(I):= \left\{\delta^{(m)}_{\mathcal{K}}(I)\ \vert\ \forall\ \mathcal{K} \text{ s.t. } \vert \mathcal{K} 
\vert = m+p(s) \text{ and } \mathcal{I} \subset \mathcal{K}\right\}.\]
\end{definition}

We need the following two results proved in \citet{BLMR} in the case $s=r'=r$. 
We observe that the quoted proofs work for ideals generated in degree $s$ if the saturation of    $J=(J_s)$ has regularity at most $r'$ and $s\geq r'$.

\begin{prop}\label{cor:costbase}\citep[Proposition 2.4]{BLMR} If  $I  \in \underline{\mathbf{G}}^{[s]}_{\mathcal{I},Id}(A)$, the set
$\mathcal{B}^{m}(I)$
 generates $\wedge^m I_s$ as an $A$-module. 
If $A=K$ and $I_s\in \GrassFunctor{p(s)}{N(s)}(K)$, then $\mathcal{B}^{m}(I)$ generates $\wedge^m I_s$.
\end{prop}

\begin{prop}\label{prop:generatorsBorel}\citep[Proposition 4.2]{BLMR}
In the above setting, let $\mathcal I\in \mathcal S^{[r',s]}_{p(t)}$ and $I\in\underline{\mathbf{G}}^{[s]}_{\mathcal{I},Id}(A)$,  $\mathcal{B}^{(1)} (I)  $  be the 
set of generators of $I_s$ given in Proposition \ref{cor:costbase}. Then:
\begin{enumerate}[(i)]
 \item\label{it:generatorsssBorel_ii} 
$\mathcal{G}^{(1)} (I_{s})  :=\bigcup_{h=0}^d  x_h \mathcal{B}^{(1)}(I _{s} )$ is a set of generators for $I^{(1)}_{s+1} $  and \\
$\mathcal{G}^{(2)} (I_{s})  :=\bigcup_{h=d+1}^n  x_h \left(\mathcal{B}^{(1)}(I _{s} ) \cap (x_0, \dots, x_d ) \right)$ 
is contained in $I^{(2)}_{s+1} $. 
 \item\label{it:generatorsBorel_iii}  We obtain a set of generators for  $I^{(2)}_{s+1} $ as the union of  $\mathcal{G}^{(2)} (I_{s})  $
 and the following set:
\begin{multline*}
\mathcal{G}^{(3)} (I_{s})  :=\left\{ x_{i} \delta^{(1)}_{\mathcal{K}}(I _{s} ) - x_{\overline{i}}\delta^{(1)}_{\overline{\mathcal{K}}}(I _{s} ) 
\in (x_0, \dots, x_d)
 \ \vert \right. \\
\left.    
i, {\overline{i}}\in\{d+1, \dots, n\},\ 
 \delta^{(1)}_{\mathcal{K}}(I _{s} ), 
\delta^{(1)}_{\overline{\mathcal{K}}}(I _{s} )   \in \mathcal{B}^{(1)}(I _{s} ) 
 \right\}.
\end{multline*}
\end{enumerate}
\end{prop}
In \cite{BLMR} the above result is proved with respect to the subset $\mathcal{B}^{(1)}_{\mathcal I}(I _{s} )$ of $\mathcal{B}^{(1)}(I _{s} )$. Hence, the result
also holds as we state it.

As in the following we look for conditions on the 
 Pl\"ucker coordinates, we will write
 $\delta^{(m)}_{\mathcal{K}}$ instead of $\delta^{(m)}_{\mathcal{K}}(I)$ meaning that we are considering the universal
 element and we are looking at its Pl\"ucker coordinates as variables.

Now we prove the second main result of the paper.

\begin{proof2}
We rewrite a large part of the proof of \citet[Theorem 4.7]{BLMR}. We also refer to  \citet[Proposition C.30]{ik} and \citet[Section 4]{BLMR} for a similar way to
get the equations for the Hilbert scheme.

Let $A$ be a $k$-algebra and consider $I\in\underline{\mathbf{G}}^{[s]}_{\mathcal{I},{Id}}(A)$. 
We characterize the subfunctor $\underline{\mathbf{H}}^{[r',s]}_{\mathcal{I},{Id}}$ 
in $\underline{\mathbf{G}}^{[s]}_{\mathcal{I},{Id}}$ as in \citet[Theorem 4.4 ]{LRFunt}.

The ideal  $I\in\underline{\mathbf{G}}^{[s]}_{\mathcal{I},{Id}}(A)$ is also an element of  $\underline{\mathbf{H}}^{[r',s]}_{\mathcal{I},{Id}}(A)$ if, and only if, 
it satisfies the conditions given in Lemma \ref{semplificazione}. We use the set of generators of Proposition \ref{prop:generatorsBorel} for $I^{(1)}_{s+1}$ and 
$I^{(2)}_{s+1}$.

In what follows we will consider ordered sequences  $\mathbf{m}=(m_0, \dots, m_d)$ of non-negative integers and write for sake of simplicity
$ \bigwedge_{i=0}^d x_i \delta^{(m_i)}_{\mathcal{K}_i}$ meaning $x_{i_1}\delta^{(m_{i_1})}_{\mathcal{K}_{i_1}}\wedge \dots \wedge x_{i_b}\delta^{(m_{i_b})}_{\mathcal{K}_{i_b}}$
where the $m_{i_j}$ are those positive in $\mathbf m$.

We obtain a set of generators of $\wedge^{q''(s+1)+1}I^{(1)}_{s+1}$, evaluating the following polynomials at $I$:
\begin{equation}\label{1eq}
\bigwedge_{i=0}^d x_i \delta^{(m_i)}_{\mathcal{K}_i} ,    \quad \forall\	\delta^{(m_i)}_{\mathcal{K}_i}\in \mathcal{B}^{(m_i)} , 
\quad    \sum_{i=0}^{d} m_i=q''(s+1)+1. 
\end{equation}

For what concerns  $\wedge^{q''(s+1)}I^{(1)}_{s+1}\wedge I''_{s+1}$, it is enough to consider a set of generators for  $\wedge^{q''(s+1)}I^{(1)}_{s+1}$ and $f$ 
in $\mathcal{G}^{(2)} (I_{s}) \cup \mathcal{G}^{(3)} (I_{s})$. Hence, we obtain them   evaluating the following polynomials at $I$:

\begin{equation}\label{2eq}
\left(\bigwedge_{i=0}^d x_i \delta^{(m_i)}_{\mathcal{K}_i}\right)\wedge x_h \delta^{(1)}_{\mathcal{K}},    \quad \forall\	\delta^{(m_i)}_{\mathcal{K}_i}\in 
\mathcal{B}^{(m_i)} , 
\quad    \sum_{i=0}^{d} m_i=q''(s+1), \quad  x_h\delta^{(1)}_{\mathcal{K}}  \in \mathcal{G}^{(2)}
\end{equation}

\begin{multline}\label{3eq}
\left( \bigwedge_{i=0}^d x_i \delta^{(m_i)}_{\mathcal{K}_i}    \right)\wedge \left(  x_{i} \delta^{(1)}_{\mathcal{K}}  
- x_{\overline{i}}\delta^{(1)}_{\overline{\mathcal{K}}}   
\right) ,  \\
\forall\	\delta^{(m_i)}_{\mathcal{K}_i}\in \mathcal{B}^{(m_i)} , 
\quad    \sum_{i=0}^{d} m_i=q''(s+1), \quad   x_{i} \delta^{(1)}_{\mathcal{K}}  
- x_{\overline{i}}\delta^{(1)}_{\overline{\mathcal{K}}} \in \mathcal{G}^{(3)}
\end{multline}

Therefore, the scheme $\mathbf{H}_{\mathcal I,Id}^{[r',s]}$ is defined in $\mathbf{G}_{\mathcal I,Id}^{[s]}$ by the vanishing of the $x$-coefficients of the 
polynomials in \eqref{1eq}, \eqref{2eq}, \eqref{3eq}. We  denote by  $\mathfrak{h}_{{Id}}$ this set of polynomials. 
Note that the $x$-coefficients of the polynomials in \eqref{1eq}   are polynomials in the  Pl\"ucker coordinates of degree $\leq d+1$, while those of \eqref{2eq}
and \eqref{3eq} have degree $\leq d+2$. Indeed, their degrees depend on how many of the integers $m_i$ in $\mathbf m$ are positive.

A very crucial point  is the fact that the above equations that we derived with respect to an open Borel subscheme 
 ${\mathbf{G}}^{[s]}_{\mathcal{I},{Id}}$ are independent on the set of indices $\mathcal{I}$. Every set of indices
 $\mathcal{I'} \in \mathcal{S}^{[r',s]}_{p(t)}$  
  gives rise to the same set of equations.
As a consequence the set of equations 
defining ${\mathbf{H}}^{[r',s]}_{\mathcal{I},{Id}}$  as a closed  
sub\-scheme of   ${\mathbf{G}}^{[s]}_{\mathcal{I},{Id}}$ formally coincide with that for every other Borel subscheme
${\mathbf{G}}^{[s]}_{\mathcal{I'},{Id}}$, with $\mathcal{I'} \in \mathcal{S}^{[r',s]}_{p(t)}$.
Indeed, by Proposition \ref{cor:costbase}, we can compute these equations at any ideal $I\in \GrassFunctor{p(s)}{N(s)}(K)$ with $K$ a field and consider them as 
elements in  the projective coordinate ring of $\GrassScheme{p(s)}{N(s)}$
in the Pl\"ucker embedding.

We denote by $\mathfrak{h}_{g}$  the set of polynomials obtained by
the action of an element $g\in \PGL(n+1)$ on $\mathfrak h_{Id}$. 
Note that, for every $\mathcal I\in \mathcal S^{[r',s]}_{p(t)}$,  $\mathfrak{h}_{g}$  
defines   $\mathbf{H}^{[r',s]}_{\mathcal{I},g}$ in $\mathbf{G}^{[s]}_{\mathcal{I},g}$ and its equations are of degree at most $d+2$,
 as the automorphism $g \in \PGL(n+1)$ induces a linear automorphism also on the Pl\"ucker coordinates, which does not modify the 
degree of the relations among them.

Finally,
 we can prove, following the lines of \citet[proof of Theorem 4.7]{BLMR},
that the ideal generated by  the union of the equations 
\[
\mathfrak{H} := \left( \bigcup_{g\in \PGL(n+1)} \mathfrak{h}_{g}\right)
\]
 gives a set of  global equations defining  $\HilbSchemeR{p(t)}{n}{r'} \subset \GrassScheme{p(s)}{N(s)}\setminus {L}^{[r',s]}_{p(t)}$. For convenience,  we  
 denote by $\mathcal{Z} $ the subscheme of
$\GrassScheme{p(s)}{N(s)}\setminus {L}^{[r',s]}_{p(t)}$ defined by $\mathfrak{H} $. 
 
Due to the noetherianity of the  coordinate ring of $\GrassScheme{p(s)}{N(s)}$  in the \Pl\ embedding, we can choose a finite set of generators  $ h_1, \dots, h_m$ 
of $\mathfrak{H}$,
 with $h_i\in \mathfrak{h}_{g_i}$. Moreover, by the invariance of $ \mathfrak{H} $ under the action of $\PGL(n+1)$, also
 $g\centerdot h_1, \dots, g\centerdot h_m$ is a set of generators, so that 
\begin{equation}\label{pochi}\mathfrak{H}  =\left( \mathfrak{h}_{gg_1}\cup \dots \cup\mathfrak{h}_{gg_m}\right)\end{equation}
 for each $g\in \PGL(n+1)$.
Therefore, for every $\mathcal I\in \mathcal S^{[r',s]}_{p(t)}$ and for every $g\in \PGL(n+1)$, we get
$$\HilbSchemeR{p(t)}{n}{r'}\cap (\mathbf{G}^{[s]}_{\mathcal I,gg_1}\cap\dots\cap \mathbf{G}^{[s]}_{\mathcal I,gg_m})=\mathcal Z\cap 
(\mathbf{G}^{[s]}_{\mathcal I,gg_1}\cap\dots\cap \mathbf{G}^{[s]}_{\mathcal I,gg_m}).$$

In order to conclude, it is sufficient to prove that for every $I\in \GrassScheme{p(s)}{N(s)}\setminus {L}^{[r',s]}_{p(t)}$, we can find suitable
${\mathcal I}\in \mathcal S^{[r',s]}_{p(t)}$ and $g \in \PGL(n+1)$ such that $I\in \mathbf{G}^{[s]}_{\mathcal I,gg_1}\cap\dots\cap \mathbf{G}^{[s]}_{\mathcal I,gg_m}$.

By Corollary \ref{ricpiccoloG}, we can find ${\mathcal I}$ and $\overline g$ such that $I\in {\mathbf{G}}^{[s]}_{{\mathcal{I}},\overline{g}}$.
 The orbit of $I$ under the action of $\PGL(n+1)$ 
 is almost completely contained in
 ${\mathbf{G}}^{[s]}_{{\mathcal{I}},\overline{g}}$; let $U$ be an open subset of $\PGL(n+1)$ such that
 $(g')^{-1}\centerdot I\in {\mathbf{G}}^{[s]}_{{\mathcal{I}},\overline{g}}$  for every $g'\in U$. Hence, for every $g'\in U$, 
 $I\in  {\mathbf{G}}^{[s]}_{{\mathcal{I}},g'\overline{g}}$.

Thus, taking a general element  $g\in \PGL(n+1)$,  the elements    $g g_1\overline{g}^{-1}, \dots, g g_m\overline{g}^{-1} $ are all contained in $ U$, so that 
$ \mathbf{G}^{[s]}_{\mathcal I,gg_1}\cap\dots\cap \mathbf{G}^{[s]}_{\mathcal I,gg_m}$ is the required  open   neighborhood of $I$.
 \end{proof2}

\section*{Acknowledgments}
The authors thank the referee for several very important suggestions that improved the quality of the paper.

\bibliographystyle{plainnat}

\begin{thebibliography}{28}
\providecommand{\natexlab}[1]{#1}
\providecommand{\url}[1]{\texttt{#1}}
\expandafter\ifx\csname urlstyle\endcsname\relax
  \providecommand{\doi}[1]{doi: #1}\else
  \providecommand{\doi}{doi: \begingroup \urlstyle{rm}\Url}\fi

\bibitem[Alonso et~al.(2009)Alonso, Brachat, and Mourrain]{ABM}
Mariemi Alonso, Jerome Brachat, and Bernard Mourrain.
\newblock The {H}ilbert scheme of points and its link with border basis.
\newblock Available at http://arxiv.org/abs/0911.3503, 2009.

\bibitem[Arbarello et~al.(2011)Arbarello, Cornalba, and Griffiths]{ACG}
Enrico Arbarello, Maurizio Cornalba, and Pillip~A. Griffiths.
\newblock \emph{Geometry of algebraic curves. {V}olume {II}}, volume 268 of
  \emph{Grundlehren der Mathematischen Wissenschaften [Fundamental Principles
  of Mathematical Sciences]}.
\newblock Springer, Heidelberg, 2011.
\newblock With a contribution by Joseph Daniel Harris.

\bibitem[Bayer(1982)]{B}
David Bayer.
\newblock \emph{The division algorithm and the {H}ilbert schemes}.
\newblock PhD thesis, Harvard University, 1982.
\newblock Ph.D. Thesis.

\bibitem[Bayer and Stillman(1987{\natexlab{a}})]{BS}
David Bayer and Michael Stillman.
\newblock A criterion for detecting {$m$}-regularity.
\newblock \emph{Invent. Math.}, 87\penalty0 (1):\penalty0 1--11,
  1987{\natexlab{a}}.

\bibitem[Bayer and Stillman(1987{\natexlab{b}})]{BS87}
David Bayer and Michael Stillman.
\newblock A theorem on refining division orders by the reverse lexicographic
  order.
\newblock \emph{Duke Math. J.}, 55\penalty0 (2):\penalty0 321--328,
  1987{\natexlab{b}}.

\bibitem[Bertone et~al.(2013{\natexlab{a}})Bertone, Cioffi, Lella, and
  Roggero]{BCLR}
Cristina Bertone, Francesca Cioffi, Paolo Lella, and Margherita Roggero.
\newblock Upgraded methods for the effective computation of marked schemes on a
  strongly stable ideal.
\newblock \emph{J. Symbolic Comput.}, 50:\penalty0 263--290,
  2013{\natexlab{a}}.

\bibitem[Bertone et~al.(2013{\natexlab{b}})Bertone, Lella, and Roggero]{BLR}
Cristina Bertone, Paolo Lella, and Margherita Roggero.
\newblock A {B}orel open cover of the {H}ilbert scheme.
\newblock \emph{J. Symbolic Comput.}, 53:\penalty0 119--135,
  2013{\natexlab{b}}.

\bibitem[Brachat et~al.(2013)Brachat, Lella, Mourrain, and Roggero]{BLMR}
Jerome Brachat, Paolo Lella, Bernard Mourrain, and Margherita Roggero.
\newblock {Extensors and the Hilbert scheme}.
\newblock Available at http://arxiv.org/abs/1104v2.2007, 2013.

\bibitem[Ciliberto and Sernesi(1989)]{CS}
Ciro Ciliberto and Edoardo Sernesi.
\newblock Families of varieties and the {H}ilbert scheme.
\newblock In \emph{Lectures on {R}iemann surfaces ({T}rieste, 1987)}, pages
  428--499. World Sci. Publ., Teaneck, NJ, 1989.

\bibitem[Cioffi and Roggero(2011)]{CR}
Francesca Cioffi and Margherita Roggero.
\newblock Flat families by strongly stable ideals and a generalization of
  {G}r\"obner\ bases.
\newblock \emph{J. Symbolic Comput.}, 46\penalty0 (9):\penalty0 1070--1084,
  2011.

\bibitem[Cioffi et~al.(2011)Cioffi, Lella, Marinari, and Roggero]{CLMR}
Francesca Cioffi, Paolo Lella, Maria~Grazia Marinari, and Margherita Roggero.
\newblock Segments and {H}ilbert schemes of points.
\newblock \emph{Discrete Mathematics}, 311\penalty0 (20):\penalty0 2238 --
  2252, 2011.

\bibitem[Eisenbud(1995)]{e2}
David Eisenbud.
\newblock \emph{Commutative algebra}, volume 150 of \emph{Graduate Texts in
  Mathematics}.
\newblock Springer-Verlag, New York, 1995.
\newblock With a view toward algebraic geometry.

\bibitem[Eisenbud(2005)]{e}
David Eisenbud.
\newblock \emph{The geometry of syzygies}, volume 229 of \emph{Graduate Texts
  in Mathematics}.
\newblock Springer-Verlag, New York, 2005.
\newblock A second course in commutative algebra and algebraic geometry.

\bibitem[Galligo(1974)]{Galligo}
Andr{\'e} Galligo.
\newblock \`{A} propos du th\'eor\`eme de pr\'eparation de {W}eierstrass.
\newblock In \emph{Fonctions de plusieurs variables complexes ({S}\'em.
  {F}ran\c cois {N}orguet, octobre 1970--d\'ecembre 1973; \`a la m\'emoire
  d'{A}ndr\'e {M}artineau)}, pages 543--579. Lecture Notes in Math., Vol. 409.
  Springer, Berlin, 1974.
\newblock Th{\`e}se de 3{\`e}me cycle soutenue le 16 mai 1973 {\`a} l'Institut
  de Math{\'e}matique et Sciences Physiques de l'Universit{\'e} de Nice.

\bibitem[Gotzmann(1978)]{gotz}
Gerd Gotzmann.
\newblock {Eine Bedingung f\"ur die Flachheit und das Hilbertpolynom eines
  graduierten Ringes.}
\newblock \emph{Math. Z.}, 158:\penalty0 61--70, 1978.

\bibitem[Green(2010)]{Green}
Mark~L. Green.
\newblock Generic initial ideals.
\newblock In \emph{Six lectures on commutative algebra}, Mod. Birkh\"auser
  Class., pages 119--186. Birkh\"auser Verlag, Basel, 2010.

\bibitem[Grothendieck(1995)]{Gro}
Alexander Grothendieck.
\newblock Techniques de construction et th{\'e}or{\`e}mes d'existence en
  g{\'e}om{\'e}trie alg{\'e}brique. {IV}. {L}es sch{\'e}mas de {H}ilbert.
\newblock In \emph{S{\'e}minaire {B}ourbaki, {V}ol.\ 6}, pages Exp.\ No.\ 221,
  249--276. Soc. Math. France, Paris, 1995.

\bibitem[Gruson et~al.(1983)Gruson, Lazarsfeld, and Peskine]{glp}
Laurent Gruson, Robert Lazarsfeld, and Christian Peskine.
\newblock On a theorem of {C}astelnuovo, and the equations defining space
  curves.
\newblock \emph{Invent. Math.}, 72\penalty0 (3):\penalty0 491--506, 1983.

\bibitem[Haiman and Sturmfels(2004)]{HaimSturm}
Mark Haiman and Bernd Sturmfels.
\newblock Multigraded {H}ilbert schemes.
\newblock \emph{J. Algebraic Geom.}, 13\penalty0 (4):\penalty0 725--769, 2004.

\bibitem[Hartshorne(1966)]{h2}
Robin Hartshorne.
\newblock Connectedness of the {H}ilbert scheme.
\newblock \emph{Inst. Hautes \'Etudes Sci. Publ. Math.}, \penalty0
  (29):\penalty0 5--48, 1966.

\bibitem[Hartshorne(1977)]{AG}
Robin Hartshorne.
\newblock \emph{Algebraic geometry}.
\newblock Springer-Verlag, New York, 1977.
\newblock Graduate Texts in Mathematics, No. 52.

\bibitem[Hirschowitz and Simpson(1996)]{hirscho}
Andr{\'e} Hirschowitz and Carlos Simpson.
\newblock La r\'esolution minimale de l'id\'eal d'un arrangement g\'en\'eral
  d'un grand nombre de points dans {${\bf P}^n$}.
\newblock \emph{Invent. Math.}, 126\penalty0 (3):\penalty0 467--503, 1996.

\bibitem[Iarrobino and Kleiman(1999)]{ik}
Anthony Iarrobino and Steve Kleiman.
\newblock \emph{The Gotzmann Theorems and the Hilbert Scheme ({A}ppendix {C} of
  {P}ower sums, {G}orenstein algebras, and determinantal loci)}, volume 1721 of
  \emph{Lecture Notes in Mathematics}.
\newblock Springer-Verlag, Berlin, 1999.

\bibitem[Lella and Roggero(2013)]{LRFunt}
Paolo Lella and Margherita Roggero.
\newblock On the functoriality of marked families.
\newblock Available at http://arXiv:1307.7657, 2013.

\bibitem[Mall(2000)]{mall}
Daniel Mall.
\newblock Connectedness of {H}ilbert function strata and other connectedness
  results.
\newblock \emph{J. Pure Appl. Algebra}, 150\penalty0 (2):\penalty0 175--205,
  2000.

\bibitem[Mumford(1966)]{m}
David Mumford.
\newblock \emph{Lectures on curves on an algebraic surface}.
\newblock With a section by G. M. Bergman. Annals of Mathematics Studies, No.
  59. Princeton University Press, Princeton, N.J., 1966.

\bibitem[Nitsure(2005)]{nit}
Nitin Nitsure.
\newblock Construction of {H}ilbert and {Q}uot schemes.
\newblock In \emph{Fundamental algebraic geometry}, volume 123 of \emph{Math.
  Surveys Monogr.}, pages 105--137. Amer. Math. Soc., Providence, RI, 2005.

\bibitem[Sernesi(2006)]{sernesi}
Edoardo Sernesi.
\newblock \emph{Deformations of algebraic schemes}, volume 334 of
  \emph{Grundlehren der Mathematischen Wissenschaften [Fundamental Principles
  of Mathematical Sciences]}.
\newblock Springer-Verlag, Berlin, 2006.

\end{thebibliography}

\end{document}